\documentclass[12pt,etds]{amsart}
\usepackage{amsmath}
\usepackage{amssymb}
\usepackage{amsfonts}
\usepackage{amsthm}
\usepackage{enumerate}
\usepackage{tikz}
\usetikzlibrary{matrix}

\usepackage{pb-diagram}

\newtheorem{theorem}{Theorem}[section]
\newtheorem{lemma}[theorem]{Lemma}
\newtheorem{prop}[theorem]{Proposition}

\theoremstyle{remark}
\newtheorem{definition}[theorem]{Definition}
\newtheorem{example}[theorem]{Example}

\newtheorem{remark}[theorem]{Remark}

\def\i{\iota}

\def\N{{\mathbb N}}
\def\O{{\mathcal{O}}}

\def\C{{\mathbb C}}

\def\R{{\mathbb R}}
\def\TT{{\mathbb T}}

\def\Z{{\mathbb Z}}

\def\a{\tilde{\alpha}}

\def\U{{\mathcal{U}}}
\def\K{{\mathcal{K}}}
\def\H{{\mathcal{H}}}
\def\T{{\mathcal{T}}}
\def\I{{\mathcal{I}}}
\def\J{{\mathcal{J}}}
\def\L{{\mathcal{L}}}

\newcommand{\clsp}{\overline{\operatorname{span}}}

\newcommand{\lt}{\operatorname{lt}}
\newcommand{\id}{\operatorname{id}}

\newcommand{\piso}{\operatorname{piso}}
\newcommand{\iso}{\operatorname{iso}}
\newcommand{\Ind}{\operatorname{Ind}}

\newcommand{\Prim}{\operatorname{Prim}}

\newcommand{\whitesquare}{\hfill $\whitesquare$\newline\vspace{0.4cm}}

\numberwithin{equation}{section}

\begin{document}

\title[The primitive ideal space of the partial-isometric crossed product]
{The primitive ideal space of the partial-isometric crossed product of a system by a single automorphism}

\author[Wicharn Lewkeeratiyutkul]{Wicharn Lewkeeratiyutkul}
\address{Department of Mathematics and Computer Science, Faculty of Science, Chulalongkorn University, Bangkok 10330, Thailand}
\email{Wicharn.L@chula.ac.th}

\author[Saeid Zahmatkesh]{Saeid Zahmatkesh}
\email[Corresponding author]{saeid.zk09@gmail.com, Saeid.Z@chula.ac.th}




\begin{abstract}
Let $(A,\alpha)$ be a system consisting of a $C^*$-algebra $A$ and an automorphism $\alpha$ of $A$. We describe the primitive ideal space of the partial-isometric crossed product $A\times_{\alpha}^{\piso}\N$ of the system by using its realization as a full corner of a classical crossed product and applying some results of Williams and Echterhoff.
\end{abstract}
\maketitle

\section{Introduction}
\label{intro}
Lindiarni and Raeburn in \cite{LR} introduced the partial-isometric crossed product of a dynamical system $(A,\Gamma^{+},\alpha)$ in which $\Gamma^{+}$ is the positive cone of a totally ordered abelian group $\Gamma$ and $\alpha$ is an action of $\Gamma^{+}$ by endomorphisms of A. Note that since the $C^*$-algebra $A$ is not necessarily unital, we require that each endomorphism $\alpha_{s}$ extends to a strictly continuous endomorphism $\overline{\alpha}_{s}$ of the multiplier algebra $M(A)$. This for an endomorphism $\alpha$ of $A$ happens if and only if there exists an approximate identity $(a_{\lambda})$ in $A$  and a projection $p\in M(A)$ such that $\alpha(a_{\lambda})$ converges strictly to $p$ in $M(A)$. We stress that if $\alpha$ is extendible, then we may not have $\overline{\alpha}(1_{M(A)})=1_{M(A)}$. A covariant representation of the system $(A,\Gamma^{+},\alpha)$ is defined for which the endomorphisms $\alpha_{s}$ are implemented by partial isometries, and the associated partial-isometric crossed product $A\times_{\alpha}^{\piso}\Gamma^{+}$ of the system is a $C^*$-algebra generated by a universal covariant representation such that there is a bijection between covariant representations of the system and nondegenerate representations of $A\times_{\alpha}^{\piso}\Gamma^{+}$. This generalizes the covariant isometric representation theory that uses isometries to represent the semigroup of endomorphisms in a covariant representation of the system (see \cite{ALNR}). The authors of \cite{LR}, in particular, studied the structure of the partial-isometric crossed product of the distinguished system $(B_{\Gamma^{+}},\Gamma^{+},\tau)$, where the action $\tau$ of $\Gamma^{+}$ on the subalgebra $B_{\Gamma^{+}}$ of $\ell^{\infty}(\Gamma^{+})$ is given by the right translation. Later, in \cite{AZ}, the authors showed that $A\times_{\alpha}^{\piso}\Gamma^{+}$ is a full corner in a subalgebra of the $C^*$-algebra $\L(\ell^{2}(\Gamma^{+})\otimes A)$ of adjointable operators on the Hilbert $A$-module $\ell^{2}(\Gamma^{+})\otimes A\simeq \ell^{2}(\Gamma^{+},A)$. This realization led them to identify the kernel of the natural homomorphism $q:A\times_{\alpha}^{\piso}\Gamma^{+}\rightarrow A\times_{\alpha}^{\iso}\Gamma^{+}$ as a full corner of the compact operators $\K(\ell^{2}(\N)\otimes A)$, when $\Gamma^{+}$ is $\N:=\Z^{+}$. So as an application, they recovered the Pimsner-Voiculescu exact sequence in \cite{PV}. Then in their subsequent work \cite{AZ2}, they proved that for an extendible $\alpha$-invariant ideal $I$ of $A$ (see the definition in \cite{Adji1}), the partial-isometric crossed product $I\times_{\alpha}^{\piso}\Gamma^{+}$ sits naturally as an ideal in $A\times_{\alpha}^{\piso}\Gamma^{+}$ such that $(A\times_{\alpha}^{\piso}\Gamma^{+})/(I\times_{\alpha}^{\piso}\Gamma^{+})\simeq A/I\times_{\a}^{\piso}\Gamma^{+}$. This is actually a generalization of \cite[Theorem 2.2]{Adji-Abbas}. They then combined these results to show that the large commutative diagram of \cite[Theorem 5.6]{LR} associated to the system $(B_{\Gamma^{+}},\Gamma^{+},\tau)$ is valid for any totally ordered abelian group, not only for subgroups of $\R$. In particular, they use this large commutative diagram for $\Gamma^{+}=\N$ to describe the ideal structure of the algebra $B_{\N}\times_{\tau}^{\piso}\N$ explicitly.

Now here we consider a system $(A,\alpha)$ consisting of a $C^*$-algebra $A$ and an automorphism $\alpha$ of $A$. So we actually have an action of the positive cone $\N=\Z^{+}$ of integers $\Z$ by automorphisms of $A$. In the present work, we want to study $\Prim(A\times_{\alpha}^{\piso}\N)$, the primitive ideal space of the partial-isometric crossed product $A\times_{\alpha}^{\piso}\N$ of the system. Since $A\times_{\alpha}^{\piso}\N$ is in fact a full corner of the classical crossed product $(B_{\Z}\otimes A)\times\Z$ (see \cite[\S 5]{AZ}), $\Prim(A\times_{\alpha}^{\piso}\N)$ is homeomorphic to $\Prim((B_{\Z}\otimes A)\times\Z)$. Therefore it is enough to describe $\Prim((B_{\Z}\otimes A)\times\Z)$. To do this, we apply the results on describing the primitive ideal space (ideal structure) of the classical crossed products from \cite{W,ECH}. So we consider the following two conditions:
\begin{itemize}
\item[(1)] when $A$ is separable and abelian;
\item[(2)] when $A$ is separable and $\Z$ acts on $\Prim A$ freely (see \S\ref{sec:pre}).
\end{itemize}
For the first condition, by applying a theorem of Williams, $\Prim((B_{\Z}\otimes A)\times\Z)$ is homeomorphic to a quotient space of $\Omega(B_{\Z})\times \Omega(A)\times\TT$, where $\Omega(B_{\Z})$ and $\Omega(A)$ are the spectrums of the $C^*$-algebras $B_{\Z}$ and $A$ respectively (recall that the dual $\hat{\Z}$ is identified with $\TT$ via the map $z\mapsto(\gamma_{z}:n\mapsto z^{n}$)). By computing $\Omega(B_{\Z})$, we parameterize the quotient space as a disjoint union, and then we precisely identify the open sets. For the second condition, we apply a result of Echterhoff which shows that $\Prim((B_{\Z}\otimes A)\times\Z)$ is homeomorphic to the quasi-orbit space of $\Prim (B_{\Z}\otimes A)=\Prim B_{\Z} \times \Prim A$ (see in \S\ref{sec:pre} that this is a quotient space of $\Prim (B_{\Z}\otimes A)$). Again by a similar argument to the first condition, we describe the quotient space and its topology precisely.

We begin with a preliminary section in which we recall the theory of the partial-isometric crossed products, and some discussions on the primitive ideal space of the classical crossed products briefly. In section \ref{sec:prim auto piso}, for a system $(A,\alpha)$ consisting of a $C^*$-algebra $A$ and an automorphism $\alpha$ of $A$, we apply the works of Williams and Echterhoff to describe $\Prim(A\times_{\alpha}^{\piso}\N)$ using the realization of $A\times_{\alpha}^{\piso}\N$ as a full corner of the classical crossed product $(B_{\Z}\otimes A)\times\Z$. As some examples, we compute the primitive ideal space of $C(\TT)\times_{\alpha}^{\piso}\N$ where the action $\alpha$ is given by rotation through the angle $2\pi\theta$ with $\theta$ rational and irrational. Moreover the description of the primitive ideal space of the Pimsner-Voiculescu Toeplitz algebra associated to the system $(A,\alpha)$ is completely obtained, as it is isomorphic to $A\times_{\alpha^{-1}}^{\piso}\N$. We also discuss necessary and sufficient conditions under which $A\times_{\alpha}^{\piso}\N$ is GCR (postliminal or type I). Finally in the last section, we discuss the primitivity and simplicity of $A\times_{\alpha}^{\piso}\N$.

\section{Preliminaries}
\label{sec:pre}
\subsection{The partial-isometric crossed product}
\label{piso cp}
A \emph{partial-isometric representation} of $\N$ on a Hilbert space $H$ is a map $V:\N\rightarrow B(H)$ such that each $V_{n}:= V(n)$ is a partial isometry, and $V_{n+m} = V_{n}V_{m}$ for all $n,m\in\N$.

A \emph{covariant partial-isometric representation} of $(A,\alpha)$ on a Hilbert space $H$ is a pair $(\pi, V)$ consisting of a nondegenerate representation $\pi: A\rightarrow B(H)$ and a partial-isometric representation $V :\N \rightarrow B(H)$ such that
\begin{equation}
\label{piso cov eq.1}
\pi(\alpha_{n}(a))=V_{n}\pi(a) V_{n}^{*}\ \ \textrm{and}\ \ V_{n}^{*}V_{n}\pi(a)=\pi(a)V_{n}^{*}V_{n}
\end{equation}
for all $a\in A$ and $n\in\N$.

Note that every system $(A, \alpha)$ admits a nontrivial covariant partial-isometric representation \cite[Example 4.6]{LR}: let $\pi$ be a nondegenerate representation of $A$ on $H$. Define $\Pi:A\rightarrow B(\ell^2(\N,H))$ by $(\Pi(a)\xi)(n)=\pi(\alpha_{n}(a))\xi(n)$. If
$$\H :=\clsp\{\xi\in \ell^2(\N,H): \xi(n) \in\overline{\pi}(\overline{\alpha}_{n}(1))H\ \textrm{for all}\ n\},$$
then the representation $\Pi$ is nondegenerate on $\H$. Now for every $m\in\N$, define $V_{m}$ on $\H$ by $(V_{m}\xi)(n)=\xi(n+m)$. Then the pair $(\Pi|_{\H},V)$ is a partial-isometric covariant representation of $(A, \alpha)$ on $\H$. One can see that if we take $\pi$ faithful, then $\Pi$ will be faithful as well, and $\H=\ell^2(\N,H)$ whenever $\overline{\alpha}(1)=1$ (e.g. when $\alpha$ is an automorphism).

\begin{definition}
\label{piso cp df}
A partial-isometric crossed product of $(A,\alpha)$ is a triple $(B,j_{A},j_{\N})$ consisting of a $C^*$-algebra $B$, a nondegenerate homomorphism $i_{A}:A\rightarrow B$, and a partial-isometric representation $i_{\N}:\N\rightarrow M(B)$ such that:
\begin{itemize}
\item[(i)] the pair $(j_{A},j_{\N})$ is a covariant representation of $(A, \alpha)$ in $B$;
\item[(ii)] for every covariant partial-isometric representation $(\pi,V)$ of $(A, \alpha)$ on a Hilbert space $H$, there exists a nondegenerate representation
$\pi\times V: B\rightarrow B(H)$ such that $(\pi\times V) \circ i_{A}=\pi$ and $(\overline{\pi\times V}) \circ i_{\N}=V$; and
\item[(iii)] the $C^*$-algebra $B$ is spanned by $\{i_{\N}(n)^{*} i_{A}(a) i_{\N}(m) : n,m \in\N, a\in A\}$.
\end{itemize}
By \cite[Proposition 4.7]{LR}, the partial-isometric crossed product of $(A,\alpha)$ always exists, and it is unique up to isomorphism. Thus we write the partial-isometric crossed product $B$ as $A\times_{\alpha}^{\piso}\N$.
\end{definition}

We recall that by \cite[Theorem 4.8]{LR}, a covariant representation $(\pi, V)$ of $(A,\alpha)$ on $H$ induces a faithful representation $\pi\times V$ of $A\times_{\alpha}^{\piso}\N$ if and only if $\pi$ is faithful on the range of $(1-V_{n}^{*}V_{n})$ for every $n > 0$ (one can actually see that it is enough to verify that $\pi$ is faithful on the range of $(1-V^{*}V)$, where $V:=V_{1}$).

\subsection{The primitive ideal space of crossed products associated to second countable locally compact transformation groups}
\label{prim ideal LCTG}
Let $\Gamma$ be a discrete group which acts on a topological space $X$. For every $x\in X$, the set $\Gamma\cdot x:=\{s\cdot x: s\in \Gamma\}$ is called the \emph{$\Gamma$-orbit} of $x$. The set $\Gamma_{x}:=\{s\in \Gamma: s\cdot x=x\}$, which is a subgroup of $\Gamma$, is called the \emph{stability group} of $x$. We say the $\Gamma$-action is \emph{free} or $\Gamma$ acts on $X$ \emph{freely} if $\Gamma_{x}=\{e\}$ for all $x\in X$. Consider a relation $\sim$ on $X$ such that for $x,y\in X$, $x\sim y$ if and only if $\overline{\Gamma\cdot x}=\overline{\Gamma\cdot y}$. One can see that this is an equivalence relation on $X$. The set of all equivalence classes equipped with the quotient topology is denoted by $\O(X)$ and called the \emph{quasi-orbit space}, which is always a $T_{0}$-topological space. The equivalence class of each $x\in X$ is denoted by $\O(x)$ and called the \emph{quasi-orbit} of $x$.

Now let $\Gamma$ be an abelian countable discrete group which acts on a second countable locally compact Hausdorff space $X$. So $(\Gamma,X)$ is a second countable locally compact transformation group with $\Gamma$ abelian. Then the associated dynamical system $(C_{0}(X),\Gamma,\tau)$ is separable with $\Gamma$ abelian, and so the primitive ideals of $C_{0}(X)\times_{\tau} \Gamma$ are known (see \cite[Theorem 8.21]{W}). Furthermore, the topology of $\Prim(C_{0}(X)\times_{\tau} \Gamma)$ has been beautifully described \cite[Theorem 8.39]{W}. So here we want to recall the discussion on $\Prim (C_{0}(X)\times_{\tau} \Gamma)$ in brief. See more in \cite{W} that this is indeed a huge and deep discussion.

Let $N$ be a subgroup of $\Gamma$. If we restrict the action $\tau$ to $N$, then we obtain a dynamical system $(C_{0}(X), N,\tau|_{N})$ with the associated crossed product $C_{0}(X)\times_{\tau|_{N}} N$. Suppose that $\textsf{X}_{N}^{\Gamma}$ is the Green's $((C_{0}(X)\otimes C_{0}(\Gamma/N))\times_{\tau\otimes\lt} \Gamma)-(C_{0}(X)\times_{\tau|_{N}} N)$-imprimitivity bimodule whose structure can be found in \cite[Theorem 4.22]{W}. If $(\pi,V)$ is a covariant representation of $(C_{0}(X), N,\tau|_{N})$, then $\Ind_{N}^{\Gamma}(\pi\times V)$ denotes the representation of $C_{0}(X)\times_{\tau} \Gamma$ induced from the representation $\pi\times V$ of $C_{0}(X)\times_{\tau|_{N}} N$ via $\textsf{X}_{N}^{\Gamma}$. Now for $x\in X$, let $\varepsilon_{x}:C_{0}(X)\rightarrow \mathbb{C}\simeq B(\mathbb{C})$ be the evaluation map at $x$ and $w$ a character of $\Gamma_{x}$. Then the pair $(\varepsilon_{x},w)$ is a covariant representation of $(C_{0}(X),\Gamma_{x},\tau|_{\Gamma_{x}})$ such that the associated representation $\varepsilon_{x}\times w$ of $C_{0}(X)\times \Gamma_{x}$ is irreducible, and hence by \cite[Proposition 8.27]{W}, $\Ind_{\Gamma_{x}}^{\Gamma}(\varepsilon_{x}\times w)$ is an irreducible representation of $C_{0}(X)\times_{\tau} \Gamma$. So $\textrm{ker}\ \big(\Ind_{\Gamma_{x}}^{\Gamma}(\varepsilon_{x}\times w)\big)$ is a primitive ideal of $C_{0}(X)\times_{\tau} \Gamma$. Note if a primitive ideal is obtained in this way, then we say it is \emph{induced from a stability group}. In fact by \cite[Theorem 8.21]{W}, all primitive ideals of $C_{0}(X)\times_{\tau} \Gamma$ are induced from stability groups. Moreover since for every $w\in \widehat{\Gamma_{x}}$ there is a $\gamma\in \widehat{\Gamma}$ such that $w=\gamma|_{\Gamma_{x}}$, every primitive ideal of $C_{0}(X)\times_{\tau} \Gamma$ is actually given by the kernel of an induced irreducible representation $\Ind_{\Gamma_{x}}^{\Gamma}(\varepsilon_{x}\times \gamma|_{\Gamma_{x}})$ correspondent to a pair $(x,\gamma)$ in $X\times \widehat{\Gamma}$. To see the description of the topology of $\Prim(C_{0}(X)\times_{\tau} \Gamma)$, first note that if $(x,\gamma)$ and $(y,\mu)$ belong to $X\times \widehat{\Gamma}$ such that $\overline{\Gamma\cdot x}=\overline{\Gamma\cdot y}$ (which implies that $\Gamma_{x}=\Gamma_{y}$) and $\gamma|_{\Gamma_{x}}=\mu|_{\Gamma_{x}}$, then by \cite[Lemma 8.34]{W},
$$\textrm{ker}\ \big(\Ind_{\Gamma_{x}}^{\Gamma}(\varepsilon_{x}\times \gamma|_{\Gamma_{x}})\big)=\textrm{ker}\ \big(\Ind_{\Gamma_{y}}^{\Gamma}(\varepsilon_{y}\times \mu|_{\Gamma_{y}})\big).$$
So define a relation on $X\times \widehat{\Gamma}$ such that $(x,\gamma)\sim (y,\mu)$ if
\begin{align}
\label{f2}
\overline{\Gamma\cdot x}=\overline{\Gamma\cdot y}\ \ \textrm{and}\ \ \gamma|_{\Gamma_{x}}=\mu|_{\Gamma_{x}}.
\end{align}
One can see that $\sim$ is an equivalence relation on $X\times \widehat{\Gamma}$. Now consider the quotient space $X\times \widehat{\Gamma}/\sim$ equipped with the quotient topology. Then we have:
\begin{theorem}\cite[Theorem 8.39]{W}
\label{Will th}
Let $(\Gamma,X)$ be a second countable locally compact transformation group with $\Gamma$ abelian. Then the map $\Phi:X\times \widehat{\Gamma} \rightarrow \Prim(C_{0}(X)\times_{\tau} \Gamma)$ defined by
$$\Phi(x,\gamma):= \ker\ \big(\Ind_{\Gamma_{x}}^{\Gamma}(\varepsilon_{x}\times \gamma|_{\Gamma_{x}})\big)$$
is a continuous and open surjection, and factors through a homeomorphism of $X\times \widehat{\Gamma}/\sim$ onto $\Prim(C_{0}(X)\times_{\tau} \Gamma)$.
\end{theorem}

\begin{remark}
\label{rmk 1}
In the theorem above, note that $\Prim(C_{0}(X)\times_{\tau} \Gamma)$ is then a second countable space. This is because as it is mentioned in \cite[Remark 8.40]{W}, the quotient map $\textsf{q}:X\times \widehat{\Gamma}\rightarrow X\times \widehat{\Gamma}/\sim$ is open. Moreover, $X$ and $\widehat{\Gamma}$ both are second countable.
\end{remark}

Theorem \ref{Will th} can be applied to see that the primitive ideal space of the rational rotation algebra is homeomorphic to $\mathbb{T}^{2}$. We skip it here and refer readers to \cite[Example 8.45]{W} for more details.

\subsection{The primitive ideal space of crossed products by free actions}
\label{free action}
Let $(A,\Gamma,\alpha)$ be a classical dynamical system with $\Gamma$ discrete. Then the system gives an action of $\Gamma$ on the spectrum $\hat{A}$ of $A$ by $s\cdot[\pi]:=[\pi\circ\alpha_{s}^{-1}]$ for every $s\in\Gamma$ and $[\pi]\in \hat{A}$ (see \cite[Lemma 2.8]{W} and \cite[Lemma 7.1]{RW}). This also induces an action of $\Gamma$ on $\Prim A$ such that $s\cdot P:=\alpha_{s}(P)$ for each $s\in\Gamma$ and $P\in \Prim A$.

Recall that if  $\pi$ is a (nondegenerate) representation of $A$ on $H$ with $\ker\pi=J$, then $\Ind\pi$ denotes the induced representation $\tilde{\pi}\times U$ of $A\times_{\alpha} \Gamma$ on $\ell^{2}(\Gamma,H)$ associated to the covariant pair $(\tilde{\pi},U)$ of $(A,\Gamma,\alpha)$ defined by
$$(\tilde{\pi}(a)\xi)(s)=\pi(\alpha_{s}^{-1}(a))\xi(s)\ \ \textrm{and}\ \ (U_{t}\xi)(s)=\xi(t^{-1}s)$$ for all every $a\in A$, $\xi\in \ell^{2}(\Gamma,H)$, and $s,t\in\Gamma$. Note that by $\Ind J$, we mean $\ker(\Ind \pi)$.

Now let $(A,\Gamma,\alpha)$ be a classical dynamical system in which $A$ is separable and $\Gamma$ is an abelian discrete countable group. If $\Gamma$ acts on $\Prim A$ freely, then each primitive ideal $\ker\pi=P$ of $A$ induces a primitive ideal of $A\times_{\alpha} \Gamma$, namely $\Ind P=\ker(\Ind \pi)$, and the description of $\Prim(A\times_{\alpha} \Gamma)$ is completely available:
\begin{theorem}\cite[Corollary 10.16]{ECH}
\label{prim free action}
Suppose in the system $(A,\Gamma,\alpha)$ that $A$ is separable and $\Gamma$ is an amenable discrete countable group. If $\Gamma$ acts on $\Prim A$ freely, then the map
$$\O(\Prim A)\rightarrow \Prim(A\times_{\alpha} \Gamma)$$
$$\O(P)\mapsto \Ind P=\ker(\Ind \pi)$$
is a homeomorphism, where $\pi$ is an irreducible representation of $A$ with $\ker\pi=P$. In particular, $A\times_{\alpha} \Gamma$ is simple if and only if every $\Gamma$-orbit is dense in $\Prim A$.
\end{theorem}

We can apply the above Theorem to see that the irrational rotation algebras are simple. Readers can refer to \cite[Example 10.18]{ECH} or \cite[Example 8.46]{W} for more details.

\section{The Primitive Ideal Space of $A\times_{\alpha}^{\piso}\N$ by automorphic action}
\label{sec:prim auto piso}
First recall that if $T$ is the isometry in $B(\ell^{2}(\N))$ such that $T(e_{n})=e_{n+1}$ on the usual orthonormal basis $\{e_n\}_{n=0}^{\infty}$ of $\ell^2(\mathbb{N})$, then we have $$\K(\ell^{2}(\N))=\clsp\{T_n(1-TT^{*})T_{m}^{*}:n,m\in\N\}.$$

Now consider a system $(A,\alpha)$ consisting of a $C^*$-algebra $A$ and an automorphism $\alpha$ of $A$. Let the triples $(A\times_{\alpha}^{\piso}\N,j_{A},v)$ and $(A\times_{\alpha}\Z,i_{A},u)$ be the partial-isometric crossed product and the classical crossed product of the system respectively. Here our goal is to describe the primitive ideal space of $A\times_{\alpha}^{\piso}\N$ and its topology completely. See in \cite{AZ} that the kernel of the natural homomorphism $q:(A\times_{\alpha}^{\piso}\N,j_{A},v)\rightarrow (A\times_{\alpha}\Z,i_{A},u)$ given by $q(v_{n}^{*}j_{A}(a)v_{m})=u_{n}^{*}i_{A}(a)u_{m}$, is isomorphic to the algebra of compact operators $\K(\ell^{2}(\N))\otimes A$. Therefore we have a short exact sequence
\begin{align}
\label{exact seq.1}
0 \longrightarrow (\K(\ell^{2}(\N))\otimes A) \stackrel{\mu}{\longrightarrow} A\times_{\alpha}^{\piso}\N
\stackrel{q}{\longrightarrow} A\times_{\alpha}\Z \longrightarrow 0,
\end{align}
where $\mu(T_n(1-TT^{*})T_{m}^{*}\otimes a)=v_{n}^{*}j_{A}(a)(1-v^{*}v)v_{m}$ for all $a\in A$ and $n,m\in\N$. So $\Prim (A\times_{\alpha}^{\piso}\N)$ as a set, is given by the sets $\Prim (\K(\ell^{2}(\N))\otimes A)$ and $\Prim \big(A\times_{\alpha}\Z)$. With no condition on the system, we do not have much information about $\Prim \big(A\times_{\alpha}\Z)$ in general. However, by \cite[Proposition 2.5]{AZ}, we do know that $\ker q\simeq \K(\ell^{2}(\N))\otimes A$ is an essential ideal of $A\times_{\alpha}^{\piso}\N$. Therefore $\Prim (\K(\ell^{2}(\N))\otimes A)$ which is homeomorphic to $\Prim A$, sits in $\Prim (A\times_{\alpha}^{\piso}\N)$ as an open dense subset. We will identify this open dense subset, namely the primitive ideals $\{\I_{P}: P\in\Prim A\}$ of $\Prim (A\times_{\alpha}^{\piso}\N)$ coming from $\Prim A$, shortly. Moreover see in \cite[\S 5]{AZ} that $A\times_{\alpha}^{\piso}\N$ is a full corner of the classical crossed product $(B_{\Z}\otimes A)\times_{\beta\otimes\alpha^{-1}}\Z$, where $B_{\mathbb{Z}}:=\clsp\{1_{n}: n\in\mathbb{Z}\}\subset \ell^{\infty}(\mathbb{Z})$, and the action $\beta$ of $\mathbb{Z}$ on $B_{\mathbb{Z}}$ is given by translation such that $\beta_{m}(1_{n})=1_{n+m}$ for all $m,n\in \mathbb{Z}$. Thus $\Prim (A\times_{\alpha}^{\piso}\N)$ is homeomorphic to $\Prim ((B_{\Z}\otimes A)\times_{\beta\otimes\alpha^{-1}}\Z)$, and hence it suffices to describe $\Prim ((B_{\Z}\otimes A)\times_{\beta\otimes\alpha^{-1}}\Z)$ and its topology. To do this, we will consider two conditions on the system that make us able to apply a theorem of Williams and a result by Echterhoff. We will also identify those primitive ideals of $A\times_{\alpha}^{\piso}\N$ coming from $\Prim \big(A\times_{\alpha}\Z)$, which form a closed subset of $\Prim (A\times_{\alpha}^{\piso}\N)$. But first, let us identify the primitive ideals $\I_{P}$.

\begin{prop}
\label{Ip}
Let $\pi:A\rightarrow B(H)$ be a nonzero irreducible representation of $A$ with $P:=\ker \pi$. If the pair $(\Pi,V)$ is defined as in \cite[Example 4.6]{LR} (see \S\ref{sec:pre}), then the associated representation of $(A\times_{\alpha}^{\piso}\N,j_{A},v)$, which we denote by $(\Pi\times V)_{P}$, is irreducible on $\ell^2(\N,H)$, and does not vanish on $\ker q\simeq\K(\ell^{2}(\N))\otimes A$.
\end{prop}

\begin{proof}
To see that $(\Pi\times V)_{P}$ is irreducible, we show that every $\xi\in\ell^2(\N,H)\backslash\{0\}$ is a cyclic vector for $(\Pi\times V)_{P}$, that is $\ell^2(\N,H)=\clsp\{(\Pi\times V)_{P}(x)(\xi):x\in(A\times_{\alpha}^{\piso}\N)\}$. We show that
\begin{align}
\label{f6}
\H:=\clsp\{(\Pi\times V)_{P}(v_{n}^{*}j_{A}(a)(1-v^{*}v)v_{m})(\xi):a\in A, n,m\in\N\}
\end{align}
equals $\ell^2(\N,H)$ which is enough. By viewing $\ell^2(\mathbb{N},H)$ as the Hilbert space $\ell^2(\mathbb{N})\otimes H$, it suffices to see that each $e_n\otimes h$ belongs to $\H$, where $\{e_n\}_{n=0}^{\infty}$ is the usual orthonormal basis of $\ell^2(\mathbb{N})$ and $h\in H$. Since $\xi\neq0$ in $\ell^2(\mathbb{N},H)$, there is $m\in\mathbb{N}$ such that $\xi(m)\neq0$ in $H$. But $\xi(m)$ is a cyclic vector for the representation $\pi:A\rightarrow B(H)$ as $\pi$ is irreducible. Thus we have $\overline{\textrm{span}}\{\pi(a)(\xi(m)):a\in A\}=H$, and hence $\textrm{span}\{e_n\otimes\big(\pi(a)\xi(m)\big):n\in\N, a\in A\}$ is dense in $\ell^2(\mathbb{N})\otimes H\simeq\ell^2(\mathbb{N},H)$. So we only have to show that $\H$ contains each element $e_n\otimes\big(\pi(a)\xi(m)\big)$. Calculation shows that
\begin{eqnarray*}
\begin{array}{rcl}
e_n\otimes\big(\pi(a)\xi(m)\big)&=&(V_n^{*}\Pi(a)(1-V^{*}V)V_m)(\xi)\\
&=&(\Pi\times V)_{P}(v_{n}^{*}j_{A}(a)(1-v^{*}v)v_{m})(\xi),
\end{array}
\end{eqnarray*}
and therefore $e_n\otimes(\pi(a)\xi(m))\in \H$ for every $a\in A$ and $n\in\N$. So we have $\H=\ell^2(\mathbb{N},H)$.

To show that $(\Pi\times V)_{P}$ does not vanish on $\K(\ell^{2}(\N))\otimes A$, first note that since $\pi$ is nonzero, $\pi(a)h\neq 0$ for some $a\in A$ and $h\in H$. Now if we take $(1-TT^{*})\otimes a\in\K(\ell^{2}(\N))\otimes A$, then $(\Pi\times V)_{P}(\mu((1-TT^{*})\otimes a))=(\Pi\times V)_{P}(j(a)(1-v^{*}v))\neq 0$. This is because for $(e_{0}\otimes h)\in\ell^2(\mathbb{N},H)$, we have
$$(\Pi\times V)_{P}(j_{A}(a)(1-v^{*}v))(e_{0}\otimes h)=\Pi(a)(1-V^{*}V)(e_{0}\otimes h)=e_{0}\otimes \pi(a)h,$$
which is not zero in $\ell^2(\N,H)$ as $\pi(a)h\neq 0$.
\end{proof}

\begin{remark}
\label{map Ip}
The primitive ideals $\I_{P}$ are actually the kernels of the irreducible representations $(\Pi\times V)_{P}$ which form the open dense subset
$$\U:=\{\I\in\Prim (A\times_{\alpha}^{\piso}\N):\K(\ell^{2}(\N))\otimes A\simeq\ker q\not\subset \I\}$$
of $\Prim \big(A\times_{\alpha}^{\piso}\N\big)$ homeomorphic to $\Prim (\K(\ell^{2}(\N))\otimes A)$. Now $\Prim (\K(\ell^{2}(\N))\otimes A)$ itself is homeomorphic to $\Prim A$ via the (Rieffel) homeomorphism $P\mapsto \K(\ell^{2}(\N))\otimes P$. But $\K(\ell^{2}(\N))\otimes P$ is the kernel of the irreducible representation $(\id\otimes\pi)$ of $\K(\ell^{2}(\N))\otimes A$, which indeed equals the restriction $(\Pi\times V)_{P}|_{\K(\ell^{2}(\N))\otimes A}$. Therefore we have
\begin{eqnarray*}
\begin{array}{rcl}
\I_{P}\cap (\K(\ell^{2}(\N))\otimes A)&=&\ker((\Pi\times V)_{P}|_{\K(\ell^{2}(\N))\otimes A})\\
&=&\ker(\id\otimes\pi)=\K(\ell^{2}(\N))\otimes P.
\end{array}
\end{eqnarray*}
Consequently the map $P\mapsto\I_{P}$ is a homeomorphism of $\Prim A$ onto the open dense subset $\U$ of $\Prim (A\times_{\alpha}^{\piso}\N)$.
\end{remark}

Now we want to describe the topology of $\Prim ((B_{\Z}\otimes A)\times_{\beta\otimes\alpha^{-1}}\Z)\simeq\Prim (A\times_{\alpha}^{\piso}\N)$ and identify the primitive ideals of $A\times_{\alpha}^{\piso}\N$ coming from $A\times_{\alpha}\Z$ under the following two conditions:
\begin{itemize}
\item[(1)] when $A$ is separable and abelian, by applying a theorem of Williams, namely Theorem \ref{Will th};
\item[(2)] when $A$ is separable and $\Z$ acts on $\Prim A$ freely, by applying Theorem \ref{prim free action}.
\end{itemize}

\subsection{The topology of $\Prim ((B_{\Z}\otimes A)\times_{\beta\otimes\alpha^{-1}}\Z)$ when $A$ is separable and abelian}
\label{case 1}
Suppose that $A$ is separable and abelian. Then $(B_{\Z}\otimes A)\times_{\beta\otimes\alpha^{-1}}\Z$ is isomorphic to the crossed product $C_{0}(\Omega(B_{\Z}\otimes A))\times_{\tau} \Z$ associated to the second countable locally compact transformation group $(\Z,\Omega(B_{\Z}\otimes A))$. Therefore by Theorem \ref{Will th}, $\Prim ((B_{\Z}\otimes A)\times_{\beta\otimes\alpha^{-1}}\Z)$ is homeomorphic to $\Omega(B_{\Z}\otimes A)\times \TT/\sim$. But we want to describe $\Omega(B_{\Z}\otimes A)\times \TT/\sim$ precisely. To do this, we need to analyze $\Omega(B_{\Z}\otimes A)$, and since $\Omega(B_{\Z}\otimes A)\simeq\Omega(B_{\Z})\times \Omega(A)$ (see \cite[Theorem B.37]{RW} or \cite[Theorem B.45]{RW}), we have to compute $\Omega(B_{\Z})$ first.

\begin{lemma}
\label{BZ}
Let $\{-\infty\} \cup \mathbb{Z} \cup \{\infty\}$ be the two-point compactification of $\mathbb{Z}$. Then $\Omega(B_{\Z})$ is homeomorphic to the open dense subset $\mathbb{Z} \cup \{\infty\}$.
\end{lemma}

\begin{proof}
First note that $B_{\mathbb{Z}}$ exactly consists of those functions $f:\mathbb{Z}\rightarrow \mathbb{C}
$ such that $\lim_{n\rightarrow -\infty}f(n)=0$ and $\lim_{n\rightarrow \infty}f(n)$ exists. Thus the complex homomorphisms (irreducible representations) of $B_{\mathbb{Z}}$ are given by the evaluation maps $\{\varepsilon_{n}:n\in \mathbb{Z}\}$, and the map $\varepsilon_{\infty}:B_{\mathbb{Z}}\rightarrow \mathbb{C}$ defined by $\varepsilon_{\infty}(f):=\lim_{n\rightarrow \infty}f(n)$ for all $f\in B_{\mathbb{Z}}$. So we have $\Omega(B_{\mathbb{Z}})=\{\varepsilon_{n}:n\in \mathbb{Z}\}\cup \{\varepsilon_{\infty}\}$. Note that the kernel of $\varepsilon_{\infty}$ is the ideal $C_{0}(\mathbb{Z})=\clsp\{1_{n}-1_{m}:n<m\in \mathbb{Z}\}$ of $B_{\mathbb{Z}}$. Now let $\{-\infty\} \cup \mathbb{Z} \cup \{\infty\}$ be the two-point compactification of $\mathbb{Z}$ which is homeomorphic to the subspace
$$X:=\{-1\}\cup\{-1+1/(1-n):n\in\mathbb{Z}, n<0\}\cup\{1-1/(1+n):n\in\mathbb{Z}, n\geq0\}\cup\{1\}$$ of $\mathbb{R}$. Then the map
$$f\in B_{\mathbb{Z}}\mapsto \tilde{f}\in C(\{-\infty\} \cup \mathbb{Z} \cup \{\infty\}),$$ where
\[
\tilde{f}(r):=
   \begin{cases}
      \lim_{n\rightarrow \infty} f(n) &\textrm{if}\empty\ \text{$r=\infty$,}\\
      f(r) &\textrm{if}\empty\ \text{$r\in\mathbb{Z}$,}\ \textrm{and}\\
      0 &\textrm{if}\empty\ \text{$r=-\infty$,}
   \end{cases}
\]
embeds $B_{\mathbb{Z}}$ in $C(\{-\infty\} \cup \mathbb{Z} \cup \{\infty\})$ as the maximal ideal $$I:=\{\tilde{f}\in C(\{-\infty\} \cup \mathbb{Z} \cup \{\infty\}): \tilde{f}(-\infty)=0\}.$$ Thus it follows that $\Omega(B_{\mathbb{Z}})$ is homeomorphic to $\hat{I}$, and $\hat{I}$ itelf is homeomorphic to the open subset
$$\{\pi\in C(\{-\infty\} \cup \mathbb{Z} \cup \{\infty\})^{\wedge}: \pi|_{I}\neq 0\}=\{\tilde{\varepsilon}_{r}: r\in(\mathbb{Z} \cup \{\infty\})\}$$ of $C(\{-\infty\} \cup \mathbb{Z} \cup \{\infty\})^{\wedge}$ in which each $\tilde{\varepsilon}_{r}$ is an evaluation map. So by the homeomorphism between $C(\{-\infty\} \cup \mathbb{Z} \cup \{\infty\})^{\wedge}$ and $\{-\infty\} \cup \mathbb{Z} \cup \{\infty\}$, the open subset $\{\tilde{\varepsilon}_{r}: r\in(\mathbb{Z} \cup \{\infty\})\}$ is homeomorphic to the open (dense) subset $\mathbb{Z} \cup \{\infty\}$ of $\{-\infty\} \cup \mathbb{Z} \cup \{\infty\}$ equipped with the relative topology. Therefore $\Omega(B_{\mathbb{Z}})$ is in fact homeomorphic to $\mathbb{Z} \cup \{\infty\}$. One can see that $\mathbb{Z} \cup \{\infty\}$ is indeed a second countable locally compact Hausdorff space with
$$\mathcal{B}:=\{\{n\}:n\in\mathbb{Z}\}\cup\{J_{n}:n\in\mathbb{Z}\}$$ as a countable basis for its topology, where $J_{n}:=\{n,n+1,n+2,...\}\cup\{\infty\}$ for every $n\in\mathbb{Z}$.
\end{proof}

\begin{remark}
\label{Z-action}
Before we continue, we need to mention that, if $A$ is a separable $C^*$-algebra (not necessarily abelian), then by \cite[Theorem B.45]{RW} and using Lemma \ref{BZ}, $(C_{0}({\Z})\otimes A)^{\widehat{}}$ and $(B_{\Z}\otimes A)^{\widehat{}}$ are homeomorphic to $\Z\times \hat{A}$ and $(\mathbb{Z} \cup \{\infty\})\times \hat{A}$ respectively. Also $\Prim(C_{0}({\Z})\otimes A)$ and $\Prim(B_{\Z}\otimes A)$ are homeomorphic to $\Z\times \Prim A$ and $(\mathbb{Z} \cup \{\infty\})\times \Prim A$ respectively (note that these homeomorphisms are $\Z$-equivariant for the action of $\Z$). Since $C_{0}({\Z})\otimes A$ is an (essential) ideal of $B_{\Z}\otimes A$, we have the following commutative diagram:
\begin{equation*}
\begin{diagram}\dgARROWLENGTH=0.5\dgARROWLENGTH
\node{\Z\times \hat{A}} \arrow{s,l}{\id}\arrow{e}\node{(C_{0}({\Z})\otimes A)^{\widehat{}}} \arrow{s,l}{\i}\arrow{e,t}{\Theta}
\arrow{s}\arrow{e}\node{\Prim(C_{0}({\Z})\otimes A)}\arrow{s,l}{\tilde{\i}}\arrow{e}{}
\node{\Z\times \Prim A}\arrow{s,l}{\id}\\
\node{(\mathbb{Z} \cup \{\infty\})\times \hat{A}} \arrow{e} \node{(B_{\Z}\otimes A)^{\widehat{}}} \arrow{e,t}{\tilde{\Theta}} \node {\Prim(B_{\Z}\otimes A)} \arrow{e}{} \node {(\mathbb{Z} \cup \{\infty\})\times \Prim A,}
\end{diagram}
\end{equation*}
where $\Theta$ and $\tilde{\Theta}$ are the canonical continuous, open surjections, and $\i$ an $\tilde{\i}$ are the canonical embedding maps. Now to see how $\Z$ acts on $(\mathbb{Z} \cup \{\infty\})\times \hat{A}$ (and accordingly on $(\mathbb{Z} \cup \{\infty\})\times \Prim A)$, note that since the crossed products $(C_{0}({\Z})\otimes A)\times_{\beta\otimes\alpha^{-1}}\Z$ and $(C_{0}({\Z})\otimes A)\times_{\beta\otimes\id}\Z$ are isomorphic (see \cite[Lemma 7.4]{W}), we have
$$n\cdot (m,[\pi])=(m+n,[\pi])\ \textrm{and}\ n\cdot(\infty,[\pi])=(n+\infty,n\cdot[\pi])=(\infty,[\pi\circ\alpha_{n}])$$
for all $n,m\in\Z$ and $[\pi]\in\hat{A}$. Accordingly
$$n\cdot (m,P)=(m+n,P)\ \textrm{and}\ n\cdot(\infty,P)=(\infty,\alpha_{n}^{-1}(P))$$
for all $n,m\in\Z$ and $P\in\Prim A$.
\end{remark}

So when $A$ is separable and abelian, using Lemma \ref{BZ}, $\Omega(B_{\Z}\otimes A)=(\mathbb{Z} \cup \{\infty\})\times \Omega(A)$. Now to describe $((\mathbb{Z} \cup \{\infty\})\times \Omega(A))\times \TT/\sim$, note that by Remark \ref{Z-action}, $\Z$ acts on $(\mathbb{Z} \cup \{\infty\})\times \Omega(A)$ as follows:
$$n\cdot (m,\phi)=(m+n,\phi)\ \textrm{and}\ n\cdot(\infty,\phi)=(\infty,\phi\circ\alpha_{n})$$
for all $n,m\in\Z$ and $\phi\in\Omega(A)$. Therefore, the stability group of each $(m,\phi)$ is $\{0\}$, and the stability group of each $(\infty,\phi)$ equals the stability group $\Z_{\phi}$ of $\phi$. Accordingly, the $\mathbb{Z}$-orbit of each $(m,\phi)$ is $\Z\times \{\phi\}$, and the $\mathbb{Z}$-orbit of $(\infty,\phi)$ is $\{\infty\}\times \Z\cdot\phi$, where $\Z\cdot\phi$ is the $\mathbb{Z}$-orbit of $\phi$. So for the pairs (or triples) $((m,\phi),z)$ and $((n,\psi),w)$ of $(\mathbb{Z}\times\Omega(A))\times \mathbb{T}$, we have
\begin{eqnarray*}
\begin{array}{rcl}
((m,\phi),z)\sim((n,\psi),w)&\Longleftrightarrow&\overline{\Z \cdot(m,\phi)}=\overline{\Z\cdot(n,\psi)}\\
&\Longleftrightarrow&\overline{\Z\times \{\phi\}}=\overline{\Z\times \{\psi\}}\\
&\Longleftrightarrow&\overline{\Z}\times \overline{\{\phi\}}=\overline{\Z}\times \overline{\{\psi\}}\\
&\Longleftrightarrow&(\mathbb{Z} \cup \{\infty\})\times \overline{\{\phi\}}=(\mathbb{Z} \cup \{\infty\})\times \overline{\{\psi\}}\\
&\Longleftrightarrow&(\mathbb{Z} \cup \{\infty\})\times \{\phi\}=(\mathbb{Z} \cup \{\infty\})\times \{\psi\}.
\end{array}
\end{eqnarray*}
The last equivalence follows from the fact that $\Omega(A)$ is Hausdorff. Therefore $((m,\phi),z)$ and $((n,\psi),w)$ are in the same equivalence class in $((\mathbb{Z} \cup \{\infty\})\times \Omega(A))\times \TT/\sim$ if and only if $\phi=\psi$, while $((m,\phi),z)\nsim ((\infty,\psi),w)$ for every $\psi\in \Omega(A)$ and $w\in \mathbb{T}$, because
$$\overline{\Z \cdot(\infty,\psi)}=\overline{\{\infty\}\times \Z\cdot\psi}=\overline{\{\infty\}}\times \overline{\Z\cdot\psi}=\{\infty\}\times \overline{\Z\cdot\psi}.$$
Thus if $\phi\in \Omega(A)$, then all pairs $((m,\phi),z)$ for every $m\in\Z$ and $z\in\TT$ are in the same equivalence class, which can be parameterized by $\phi\in\Omega(A)$. On the other hand, for the pairs $((\infty,\phi),z)$ and $((\infty,\psi),w)$, we have
\begin{eqnarray*}
\begin{array}{rcl}
((\infty,\phi),z)\sim((\infty,\psi),w)&\Longleftrightarrow&\overline{\Z \cdot(\infty,\phi)}=\overline{\Z\cdot(\infty,\psi)}\ \textrm{and}\ \gamma_{z}|_{\Z_{\phi}}=\gamma_{w}|_{\Z_{\phi}}\\
&\Longleftrightarrow&\{\infty\}\times \overline{\Z\cdot\phi}=\{\infty\}\times \overline{\Z\cdot\psi}\ \textrm{and}\ \gamma_{z}|_{\Z_{\phi}}=\gamma_{w}|_{\Z_{\phi}}.
\end{array}
\end{eqnarray*}
Therefore
$$((\infty,\phi),z)\sim((\infty,\psi),w)\Longleftrightarrow\overline{\Z\cdot\phi}=\overline{\Z\cdot\psi}\ \textrm{and}\ \gamma_{z}|_{\Z_{\phi}}=\gamma_{w}|_{\Z_{\phi}},$$
which means if and only if the pairs $(\phi,z)$ and $(\psi,w)$ are in the same equivalence class in the quotient space $\Omega(A)\times \TT/\sim$ homeomorphic to $\Prim \big(A\times_{\alpha}\Z)$. Therefore $((\infty,\phi),z)\sim((\infty,\psi),w)$ in $((\mathbb{Z} \cup \{\infty\})\times \Omega(A))\times \TT/\sim$ precisely when $(\phi,z)\sim(\psi,w)$ in $\Omega(A)\times \TT/\sim$, and hence the class of each $((\infty,\phi),z)$ in $((\mathbb{Z} \cup \{\infty\})\times \Omega(A))\times \TT/\sim$ can be parameterized by the class of $(\phi,z)$ in $\Omega(A)\times \TT/\sim$. So we can identify $((\mathbb{Z} \cup \{\infty\})\times \Omega(A))\times \TT/\sim$ with the disjoint union
$$\Omega(A)\sqcup (\Omega(A)\times \TT/\sim).$$ Now we have:

\begin{theorem}
\label{prim piso}
Let $(A,\alpha)$ be a system consisting of a separable abelian $C^{*}$-algebra $A$ and an automorphism $\alpha$ of $A$. Then $\Prim (A\times_{\alpha}^{\piso}\N)$ is homeomorphic to $\Omega(A)\sqcup (\Omega(A)\times \TT/\sim)$, equipped with the (quotient) topology in which the open sets are of the form
$$\{U\subset \Omega(A):U\ \textrm{is open in}\ \Omega(A)\}\cup$$
$$\{U\cup W: U\ \textrm{is a nonempty open subset of}\ \Omega(A),\textrm{and}\ W\ \textrm{is open in}\ (\Omega(A)\times \TT/\sim)\}.$$
\end{theorem}

\begin{proof}
Since the quotient map $\textsf{q}:((\mathbb{Z} \cup \{\infty\})\times \Omega(A))\times \TT\rightarrow\Omega(A)\sqcup (\Omega(A)\times \TT/\sim)$ is open, as well as $\tilde{\textsf{q}}:\Omega(A)\times \TT\rightarrow\Omega(A)\times \TT/\sim$, for every $n\in\Z$, every open subset $O$ of $\Omega(A)$, and every open subset $V$ of $\TT$, the forward image of open subsets $\{n\}\times O\times V$ and $J_{n}\times O\times V$ by $\textsf{q}$, forms a basis for the topology of $\Omega(A)\sqcup (\Omega(A)\times \TT/\sim)$, which is
$$\{O\subset \Omega(A):O\ \textrm{is open in}\ \Omega(A)\}\cup$$
$$\{O\cup \tilde{\textsf{q}}(O\times V): O\ \textrm{is a nonempty open subset of}\ \Omega(A),\textrm{and}\ V\ \textrm{is open in}\ \TT\}.$$
As the open subsets $\tilde{\textsf{q}}(O\times V)$ also form a basis for the quotient topology of $\Omega(A)\times \TT/\sim$, we can see that each open subset of $\Omega(A)\sqcup (\Omega(A)\times \TT/\sim)$ is either an open subset $U$ of $\Omega(A)$ or of the form $U\cup W$ for some nonempty open subset $U$ in $\Omega(A)$ and some open subset $W$ in $\Omega(A)\times \TT/\sim$.
\end{proof}

\begin{remark}
\label{rmk 1}
Under the condition of Theorem \ref{prim piso}, the primitive ideals of $\Prim(A\times_{\alpha}^{\piso}\N)$ coming from $\Prim(A\times_{\alpha}\Z)$, which form the closed subset
$$\mathcal{F}:=\{\J\in\Prim (A\times_{\alpha}^{\piso}\N):\K(\ell^{2}(\N))\otimes A\simeq\ker q\subset\J\},$$ are the kernels of the irreducible representations $(\Ind_{\Z_{\phi}}^{\Z}(\phi \times \gamma_{z}|_{\Z_{\phi}}))\circ q$ corresponding to the equivalence classes of the pairs $(\phi,z)$ in $\Omega(A)\times \TT/\sim$ (again by using Theorem \ref{Will th}). Therefore if $\J_{[(\phi,z)]}$ denotes $\ker (\Ind_{\Z_{\phi}}^{\Z}(\phi \times \gamma_{z}|_{\Z_{\phi}})\circ q)$, then $\mathcal{F}=\{\J_{[(\phi,z)]}: \phi\in\Omega(A), z\in\TT\}$, and the map $[(\phi,z)]\mapsto \J_{[(\phi,z)]}$ is homeomorphism of $\Prim(A\times_{\alpha}\Z)\simeq\Omega(A)\times \TT/\sim$ onto $\mathcal{F}$.
\end{remark}

\begin{prop}
\label{GCR piso}
Let $(A,\alpha)$ be a system consisting of a separable abelian $C^*$-algebra $A$ and an automorphism $\alpha$ of $A$. Then $A\times_{\alpha}^{\piso}\N$ is GCR if and only if $\Z\backslash\Omega(A)$ is a $T_{0}$ space.
\end{prop}

\begin{proof}
By \cite[Theorem 5.6.2]{Murphy}, $A\times_{\alpha}^{\piso}\N$ is GCR if and only if $\K(\ell^{2}(\N))\otimes A\simeq\ker q$ and $A\times_{\alpha}\Z\simeq C_{0}(\Omega(A))\times_{\tau}\Z$ are GCR. But since $A$ is abelian, $\K(\ell^{2}(\N))\otimes A$ is automatically CCR, and hence it is GCR. Therefore $A\times_{\alpha}^{\piso}\N$ is GCR precisely when $A\times_{\alpha}\Z$ is GCR. By \cite[Theorem 8.43]{W}, $A\times_{\alpha}\Z$ is GCR if and only if $\Z\backslash\Omega(A)$ is $T_{0}$.
\end{proof}

\begin{prop}
\label{CCR piso}
Let $(A,\alpha)$ be a system consisting of a separable abelian $C^*$-algebra $A$ and an automorphism $\alpha$ of $A$. Then $A\times_{\alpha}^{\piso}\N$ is not CCR.
\end{prop}

\begin{proof}
Note that $A\times_{\alpha}^{\piso}\N$ is CCR if and only if $(B_{\Z}\otimes A)\times_{\beta\otimes\alpha^{-1}}\Z\simeq C_{0}(\Omega(B_{\Z}\otimes A))\times_{\tau} \Z$ is CCR, because they are Morita equivalent (see \cite[Proposition I.43]{W}). Since for the $\mathbb{Z}$-orbit of a pair $(m,\phi)$, we have
$$\overline{\Z\cdot(m,\phi)}=\overline{\Z\times \{\phi\}}=\overline{\Z}\times \overline{\{\phi\}}=(\mathbb{Z} \cup \{\infty\})\times \{\phi\},$$
it follows that $\mathbb{Z}$-orbit of $(m,\phi)$ is not closed in $\Omega(B_{\Z}\otimes A)=(\mathbb{Z} \cup \{\infty\})\times \Omega(A)$. Therefore by \cite[Theorem 8.44]{W}, $C_{0}(\Omega(B_{\Z}\otimes A))\times_{\tau} \Z$ is not CCR, and hence $A\times_{\alpha}^{\piso}\N$ is not CCR.
\end{proof}

\begin{example}(Pimsner-Voiculescu Toeplitz algebra)
\label{PV algebra}
Suppose $\T(A,\alpha)$ is the Pimsner-Voiculescu Toeplitz algebra associated to the system $(A,\alpha)$ (see \cite{PV}). It was shown in \cite[\S 5]{AZ} that $\T(A,\alpha)$ is isomorphic to the partial-isometric crossed product $A\times_{\alpha^{-1}}^{\piso}\N$ associated to the system $(A,\alpha^{-1})$. Therefore when $A$ is abelian and separable, the description of $\Prim(\T(A,\alpha))$ follows completely from Theorem \ref{prim piso}. In particular, for the trivial system $(\C,\id)$, $\T(\C,\id)$ is the Toeplitz algebra $\T(\Z)$ of integers isomorphic to $\C\times_{\id}^{\piso}\N$. So again by Theorem \ref{prim piso}, $\Prim(\T(\Z))$ corresponds to the disjoint union
$\{0\}\sqcup \TT$ in which every (nonempty) open set is of the form $\{0\} \cup W$ for some open subset $W$ of $\mathbb{T}$. This description is known which coincides with the description of $\Prim(\T(\Z))$ obtained from the well-known short exact sequence $0\rightarrow\K(\ell^{2}(\N))\rightarrow \T(\Z)\rightarrow C(\TT)\rightarrow 0$.
\end{example}

\begin{example}
\label{rational}
Consider the system $(C(\TT),\alpha)$ in which the action $\alpha$ is given by rotation through the angle $2\pi\theta$ with $\theta$
rational. By using the discussion in \cite[Example 8.46]{W}, $\Prim(C(\TT)\times_{\alpha}^{\piso}\N)$ can be identified with the disjoint union
$$\TT \sqcup \TT^{2},$$ in which by Theorem \ref{prim piso}, each open set is given by
$$\{U\subset \TT:U\ \textrm{is open in}\ \TT\}\cup$$
$$\{U\cup W: U\ \textrm{is a nonempty open subset of}\ \TT,\textrm{and}\ W\ \textrm{is open in}\ \TT^{2}\}.$$
Moreover the orbit space $\Z\backslash\ \TT$ is homeomorphic to $\TT$, which is obviously $T_{0}$ (in fact Hausdorff). So it follows by Proposition \ref{GCR piso} that $C(\TT)\times_{\alpha}^{\piso}\N$ is GCR.
\end{example}

\subsection{The topology of $\Prim ((B_{\Z}\otimes A)\times_{\beta\otimes\alpha^{-1}}\Z)$ when $A$ is separable and $\Z$ acts on $\Prim A$ freely}
\label{case 2}
Consider a system $(A,\alpha)$ in which $A$ is separable, and $\Z$ acts on $\Prim A$ freely. It follows that $\Z$ acts on $\Prim(B_{\Z}\otimes A)$ freely too. This is because, firstly, by \cite[Theorem B.45]{RW}, $\Prim(B_{\Z}\otimes A)$ is homeomorphic to $\Prim B_{\Z}\times \Prim A$, and hence it is homeomorphic to $(\mathbb{Z} \cup \{\infty\})\times \Prim A$. Then $\Z$ acts on $(\mathbb{Z} \cup \{\infty\})\times \Prim A$ such that
$$n\cdot (m,P)=(m+n,P)\ \textrm{and}\ n\cdot(\infty,P)=(\infty,\alpha_{n}^{-1}(P))$$
for all $n,m\in\Z$ and $P\in\Prim A$. Therefore the stability group of each $(\infty,P)$ equals the stability group $\Z_{P}$ of $P$, which is $\{0\}$ as $\Z$ acts on $\Prim A$ freely, and stability group of each $(m,P)$ is clearly $\{0\}$. So in the separable system $(B_{\Z}\otimes A,\Z,\beta\otimes\alpha^{-1})$ (with $\Z$ abelian), $\Z$ acts on $\Prim(B_{\Z}\otimes A)\simeq (\mathbb{Z} \cup \{\infty\})\times \Prim A$ freely. Therefore by Theorem \ref{prim free action}, $\Prim ((B_{\Z}\otimes A)\times_{\beta\otimes\alpha^{-1}}\Z)$ is homeomorphic to the quasi-orbit space $\O(\Prim (B_{\Z}\otimes A))=\O((\mathbb{Z} \cup \{\infty\})\times \Prim A)$, which describes $\Prim(A\times_{\alpha}^{\piso}\N)$ as well. We want to describe the quotient topology of $\O((\mathbb{Z} \cup \{\infty\})\times \Prim A)$ precisely, and identify the primitive ideals of $A\times_{\alpha}^{\piso}\N$ coming from $\Prim(A\times_{\alpha}\Z)$. We have
\begin{eqnarray*}
\begin{array}{rcl}
\O(m,P)=\O(n,Q)&\Longleftrightarrow&\overline{\Z \cdot(m,P)}=\overline{\Z\cdot(n,Q)}\\
&\Longleftrightarrow&\overline{\Z\times \{P\}}=\overline{\Z\times \{Q\}}\\
&\Longleftrightarrow&\overline{\Z}\times \overline{\{P\}}=\overline{\Z}\times \overline{\{Q\}}\\
&\Longleftrightarrow&(\mathbb{Z} \cup \{\infty\})\times \overline{\{P\}}=(\mathbb{Z} \cup \{\infty\})\times \overline{\{Q\}}.
\end{array}
\end{eqnarray*}
Therefore $\O(m,P)=\O(n,Q)$ if and only if $\overline{\{P\}}=\overline{\{Q\}}$, and this happens precisely when $P=Q$ by the definition of the hull-kernel (Jacobson) topology on $\Prim A$ (that is why the primitive ideal space of any $C^{*}$-algebra is always $T_{0}$ \cite[Theorem 5.4.7]{Murphy}). So all pairs $(m,P)$ for every $m\in\Z$ have the same quasi-orbit which can be parameterized by $P\in \Prim A$, and since
$$\overline{\Z\cdot(\infty,Q)}=\overline{\{\infty\}\times \Z\cdot Q}=\overline{\{\infty\}}\times \overline{\Z\cdot Q}=\{\infty\}\times \overline{\Z\cdot Q},$$
$\O(m,P)\neq \O(\infty,Q)$ for all $m\in\Z$ and $P,Q\in\Prim A$. Moreover
\begin{eqnarray*}
\begin{array}{rcl}
\O(\infty,P)=\O(\infty,Q)&\Longleftrightarrow&\overline{\Z\cdot(\infty,P)}=\overline{\Z\cdot(\infty,Q)}\\
&\Longleftrightarrow&\{\infty\}\times \overline{\Z\cdot P}=\{\infty\}\times \overline{\Z\cdot Q}.

\end{array}
\end{eqnarray*}
Thus $\O(\infty,P)=\O(\infty,Q)$ if and only if $\overline{\Z\cdot P}=\overline{\Z\cdot Q}$, which means if and only if $P$ and $Q$ have the same quasi-orbit ($\O(P)=\O(Q)$) in $\O(\Prim A)\simeq \Prim(A\times_{\alpha}\Z)$. So each quasi-orbit $\O(\infty,P)$ can be parameterized by the quasi-orbit $\O(P)$ in $\O(\Prim A)$, and we can therefore identify $\O((\mathbb{Z} \cup \{\infty\})\times \Prim A)$ by the disjoint union
$$\Prim A \sqcup \O(\Prim A ).$$ Then we have:
\begin{theorem}
\label{prim piso2}
Let $(A,\alpha)$ be a system consisting of a separable $C^{*}$-algebra $A$ and an automorphism $\alpha$ of $A$. Suppose that $\Z$ acts on $\Prim A$ freely. Then $\Prim (A\times_{\alpha}^{\piso}\N)$ is homeomorphic to $\Prim A \sqcup \O(\Prim A )$, equipped with the (quotient) topology in which the open sets are of the form
$$\{U\subset \Prim A:U\ \textrm{is open in}\ \Prim A\}\cup$$
$$\{U\cup W: U\ \textrm{is a nonempty open subset of}\ \Prim A,\textrm{and}\ W\ \textrm{is open in}\ \O(\Prim A )\}.$$
\end{theorem}

\begin{proof}
Note that since by \cite[Lemma 6.12]{W}, the quasi-orbit map $\textsf{q}:\Prim(B_{\Z}\otimes A)\rightarrow\O(\Prim(B_{\Z}\otimes A))$ is continuous and open,  the proof follows from a similar argument to the proof of Theorem \ref{prim piso}. So we skip it here.
\end{proof}

\begin{remark}
\label{rmk 2}
Under the condition of Theorem \ref{prim piso2}, we want to identify the primitive ideals of $\Prim(A\times_{\alpha}^{\piso}\N)$ coming from $\Prim(A\times_{\alpha}\Z)$, which form the closed subset $$\mathcal{F}:=\{\J\in\Prim (A\times_{\alpha}^{\piso}\N):\K(\ell^{2}(\N))\otimes A\simeq\ker q\subset\J\}$$ homeomorphic to $\Prim (A\times_{\alpha}\Z)\simeq\O(\Prim A)$ (see Theorem \ref{prim free action}). These ideals are actually the kernels of the irreducible representations $(\Ind \pi)\circ q=(\tilde{\pi}\times U)\circ q$ of $A\times_{\alpha}^{\piso}\N$, where $\pi$ is an irreducible representation of $A$ with $\ker\pi=P$. But since the pair $(\tilde{\pi},U)$ is clearly a covariant partial-isometric representation of $(A,\alpha)$, one can see that in fact, $(\Ind \pi)\circ q=\tilde{\pi}\times^{\piso} U$, where $\tilde{\pi}\times^{\piso} U$ is the associated representation of $A\times_{\alpha}^{\piso}\N$ correspondent to the pair $(\tilde{\pi},U)$. Thus each element of $\mathcal{F}$ is of the form $\ker (\tilde{\pi}\times^{\piso} U)$ correspondent to the quasi-orbit $\O(P)$, and therefore we denote $\ker (\tilde{\pi}\times^{\piso} U)$ by $\J_{\O(P)}$. So the map $\O(P)\rightarrow\J_{\O(P)}$ is a homeomorphism of $\O(\Prim A)$ onto the closed subspace $\mathcal{F}$ of $\Prim \big(A\times_{\alpha}^{\piso}\N\big)$.
\end{remark}

For the following remark, we need to recall that the primitive ideal space of any $C^{*}$-algebra $A$ is locally compact \cite[Corollary 3.3.8]{Dix}. A locally compact space $X$ (not necessarily Hausdorff) is called \emph{almost Hausdorff} if each locally compact subspace $U$ contains a relatively open nonempty Hausdorff subset (see \cite[Definition 6.1.]{W}). If a $C^{*}$-algebra is GCR, then it is almost Hausdorrff (see the discussion on pages 171 and 172 of \cite{W}). Finally if $A$ is separable, then by applying \cite[Theorem A.38]{RW} and \cite[Proposition A.46]{RW}, it follows that $\Prim A$ is second countable.

\begin{remark}
\label{ZM result}
It follows from \cite{ZM} that if $(A,\Z,\alpha)$ is a separable system in which  $\Z$ acts on $\hat{A}$ freely, then $A\times_{\alpha}\Z$ is GCR if and only if $A$ is GCR and every $\Z$-orbit in $\hat{A}$ is discrete. But every $\Z$-orbit in $\hat{A}$ is discrete if and only if for each $[\pi]\in\hat{A}$, the map $\Z\rightarrow \Z\cdot[\pi]$ defined by $n\mapsto n\cdot[\pi]=[\pi\circ\alpha_{n}^{-1}]$ is a homeomorphism, and this statement itself, by \cite[Theorem 6.2 (Mackey-Glimm Dichotomy)]{W}, is equivalent to saying that the orbit space $\Z\backslash\hat{A}$ is $T_{0}$. Therefore we can rephrase the statement of \cite{ZM} to say that if $(A,\Z,\alpha)$ is a separable system in which $\Z$ acts on $\hat{A}$ freely, then $A\times_{\alpha}\Z$ is GCR if and only if $A$ is GCR and the orbit space $\Z\backslash\hat{A}$ is $T_{0}$.
\end{remark}

\begin{prop}
\label{GCR piso2}
Let $(A,\alpha)$ be a system consisting of a separable $C^*$-algebra $A$ and an automorphism $\alpha$ of $A$. Suppose that $\Z$ acts on $\hat{A}$ freely. Then $A\times_{\alpha}^{\piso}\N$ is GCR if and only if $A$ is GCR and the orbit space $\Z\backslash\hat{A}$ is $T_{0}$.
\end{prop}

\begin{proof}
The proof follows from a similar argument to the proof of Proposition \ref{GCR piso} and Remark \ref{ZM result}.
\end{proof}

\begin{example}
\label{irrational}
Consider the system $(C(\TT),\alpha)$ in which the action $\alpha$ is given by rotation through the angle $2\pi\theta$ with $\theta$
irrational. Then $\Z$ acts on $\Prim(C(\TT))=C(\TT)^{\widehat{}}=\TT$ freely (see \cite[Example 8.45]{W} or \cite[Example 10.18]{ECH}). Therefore by Theorem \ref{prim piso2}, $\Prim(C(\TT)\times_{\alpha}^{\piso}\N)$ can be identified with the disjoint union $\TT \sqcup \O(\TT)$. But the quasi-orbit space $\O(\TT)$ contains only one point as each $\Z$-orbit is dense in $\TT$ (see \cite[Lemma 3.29]{W}). Let us parameterize this only point by $0$ (note that $\O(\TT)$ is homeomorphic to the primitive ideal space of the irrational rotation algebra $A_{\theta}:=C(\TT)\times_{\alpha}\Z$ which is simple). So $\Prim(C(\TT)\times_{\alpha}^{\piso}\N)$ is actually identified with $$\TT \sqcup \{0\},$$ where each open set is given by
$$\{U\subset \TT:U\ \textrm{is open in}\ \TT\}\cup \{U\cup \{0\}: U\ \textrm{is a nonempty open subset of}\ \TT\}.$$
Here we would like to mention that $0$ in $\TT \sqcup \{0\}$ corresponds to the primitive ideal $\K(\ell^{2}(\N))\otimes C(\TT)$ of $C(\TT)\times_{\alpha}^{\piso}\N$.
Finally, although $C(\TT)$ is GCR (in fact CCR), the orbit space $\Z\backslash\TT$ is not $T_{0}$ as each $\Z$-orbit is dense in $\TT$. So it follows by Proposition \ref{GCR piso2} that $C(\TT)\times_{\alpha}^{\piso}\N$ is not GCR.
\end{example}

\begin{remark}
\label{PV algebra 2}
Recall that since the Pimsner-Voiculescu Toeplitz algebra $\T(A,\alpha)$ is isomorphic to $A\times_{\alpha^{-1}}^{\piso}\N$ (see Example \ref{PV algebra}), if $A$ is separable and $\Z$ acts on $\Prim A$ freely, then the description of $\Prim(\T(A,\alpha))$ is obtained completely from Theorem \ref{prim piso2}.
\end{remark}

\section{Primitivity and Simplicity of $A\times_{\alpha}^{\piso}\N$}
\label{sec:last}
In this section, we want to discuss the primitivity and simplicity of $A\times_{\alpha}^{\piso}\N$. Recall that a $C^{*}$-algebra is called \emph{primitive} if it has a faithful nonzero irreducible representation, and it is called \emph{simple} if it has no nontrivial ideal.

\begin{theorem}
\label{primitive}
Let $(A,\alpha)$ be a system consisting of a $C^{*}$-algebra $A$ and an automorphism $\alpha$ of $A$. Then $A\times_{\alpha}^{\piso}\N$ is primitive if and only $A$ is primitive.
\end{theorem}

\begin{proof}
If $A\times_{\alpha}^{\piso}\N$ is primitive, it has a faithful nonzero irreducible representation $\rho:A\times_{\alpha}^{\piso}\N\rightarrow B(\H)$. Then since the restriction of $\rho$ to the ideal $\K(\ell^{2}(\N))\otimes A\simeq\ker q$ is nonzero, it gives an irreducible representation of $\K(\ell^{2}(\N))\otimes A$ which is clearly faithful. So it follows that $\K(\ell^{2}(\N))\otimes A$ is primitive, and therefore $A$ must be primitive as well.

Conversely, if $A$ is primitive, then it has a faithful nonzero irreducible representation $\pi$ on some Hilbert space $H$ ($P=\ker \pi=\{0\}$). We show that the associated irreducible representation $(\Pi\times V)_{P}$ of $A\times_{\alpha}^{\piso}\N$ on $\ell^{2}(\N,H)$ is faithful. By \cite[Theorem 4.8]{LR}, it is enough to see that if $\Pi(a)(1-V^{*}V)=0$, then $a=0$. If $\Pi(a)(1-V^{*}V)=0$, then
$$\Pi(a)(1-V^{*}V)(e_{0}\otimes h)=(e_{0}\otimes \pi(a)h)=0\ \ \textrm{for all}\ h\in H.$$ It follows that $\pi(a)h=0$ for all $h\in H$, and therefore $\pi(a)=0$. Since $\pi$ is faithful, we must have $a=0$. This completes the proof.
\end{proof}

\begin{remark}
Note that Theorem \ref{primitive} simply means that in the homeomorphism $P\mapsto \I_{P}$ mentioned in Remark \ref{map Ip}, $P$ is the zero ideal if and only if $\I_{P}$ is the zero ideal. This is because if $A\times_{\alpha}^{\piso}\N$ is primitive, then its zero ideal as one of its primitive ideals is of the form $\I_{P}$ (coming from $\Prim A$), as $\K(\ell^{2}(\N))\otimes A\neq 0$.
\end{remark}

Finally it is not difficult to see that $A\times_{\alpha}^{\piso}\N$ is not simple. This is because as we see, it contains $\K(\ell^{2}(\N))\otimes A$ as a nonzero ideal. Moreover if $\K(\ell^{2}(\N))\otimes A=A\times_{\alpha}^{\piso}\N$, then $A\times_{\alpha}\Z\simeq (A\times_{\alpha}^{\piso}\N)/(\K(\ell^{2}(\N))\otimes A$) must be the zero algebra. So it follows that $A=0$, which is a contradiction as we have $A\neq0$. Therefore $A\times_{\alpha}^{\piso}\N$ contains $\K(\ell^{2}(\N))\otimes A$ as a proper nonzero ideal, and hence we have proved the following:
\begin{theorem}
\label{simple}
Let $(A,\alpha)$ be a system consisting of a $C^{*}$-algebra $A$ and an automorphism $\alpha$ of $A$. Then $A\times_{\alpha}^{\piso}\N$ is not simple.
\end{theorem}

\subsection*{Acknowledgements}
This research is supported by Rachadapisek Sompote Fund for Postdoctoral Fellowship, Chulalongkorn University.

\end{document}